\documentclass[reqno,12pt]{article}

\usepackage{amsmath} 
\usepackage{amssymb}
\usepackage{amsthm}
\usepackage{a4wide}
\usepackage[utf8]{inputenc} 
\usepackage[english]{babel}
\usepackage{xcolor}

\numberwithin{equation}{section}

\newtheorem{theo}{Theorem}
\newtheorem*{theo*}{Theorem}
\newtheorem{prop}{Proposition}

\newtheorem{coro}{Corollary}

\newtheorem{lem}{Lemma}
\newtheorem{defi}{Definition}
\newtheorem*{defi*}{Definition}

\newtheorem{thmx}{Theorem}

\theoremstyle{remark}

\newtheorem*{Remarks*}{Remarks}
\newtheorem*{Remark*}{Remark}

\newcommand{\Z}{\mathbb{Z}}
\newcommand{\Q}{\mathbb{Q}}

\newcommand{\Qbar}{\overline{\mathbb Q}}

\begin{document}

\selectlanguage{english}

\title{On Abel's problem and Gauss congruences}
\date\today
\author{\'E. Delaygue and T. Rivoal}
\maketitle

\selectlanguage{english}

\begin{abstract} A classical problem due to Abel is to determine if a differential equation $y'=\eta y$ admits a non-trivial solution $y$ algebraic over $\mathbb C(x)$ when $\eta$ is a given algebraic function over $\mathbb C(x)$. Risch designed an algorithm that, given $\eta$, determines whether there exists an algebraic solution or not. In this paper, we adopt a different point of view when $\eta$ admits a Puiseux expansion with  {\em rational} coefficients at some point in $\mathbb C\cup \{\infty\}$, which can be assumed to be 0 without loss of generality. We prove the following arithmetic characterization:  there exists a non-trivial algebraic solution of  $y'=\eta y$ if and only if the coefficients of the Puiseux expansion of $x\eta(x)$ at $0$ satisfy Gauss congruences for almost all prime numbers. We then apply our criterion to hypergeometric series: we completely determine the equations $y'=\eta y$ with an algebraic solution when $x\eta(x)$ is an algebraic hypergeometric series  with rational parameters, and this enables us to prove a prediction Golyshev made using the theory of motives. We also present 
two other applications, namely to diagonals of rational fractions and to directed two-dimensional walks.
\end{abstract}

\section{Introduction}

Given an algebraic function $\eta$ over $\mathbb{C}(x)$, Abel's problem, as mentionned by Boulanger in \cite[p. 93]{Boulanger}, consists in determining if the differential equation
\begin{equation}\label{eq: EqDiff}
y'=\eta y    
\end{equation}
admits a non-trivial solution $y$ algebraic over $\mathbb{C}(x)$. This problem naturally occurs in procedures to decide in a finite number of steps if all solutions of a linear differential equation with polynomial coefficients are algebraic over $\mathbb{C}(x)$, see \cite{BD, Singer}. A decision procedure to solve Abel's problem was first given by Risch \cite{Risch} and later independently by Baldassarri and Dwork \cite[Section 6]{BD}. 

In this paper, we are interested in the restriction of Abel's problem to the functions $\eta$ algebraic over $\mathbb{C}(x)$ that admit a Puiseux expansion with  {\em rational} coefficients at some point $\delta$ in $\mathbb C \cup\{\infty\}$. 
\medskip

Obviously the above mentioned decision procedures also apply in this context but we present here an arithmetic criterion based on Gauss congruences for $x\eta(x)$ that proves effective in applications. It is stated as Theorem \ref{theo:super} in Section~\ref{sec:Crit}, where we use in particular a consequence of Grothendieck's conjecture proved in rank one by Chudnovsky and Chudnovsky~\cite{chud}. 

Secondly, we apply our criterion to give in Section \ref{sec:Hyp} a complete resolution --~Theorem~\ref{theo:HypCrit}~-- of Abel's problem when $x\eta(x)$ is an algebraic hypergeometric series with rational parameters. We obtain a criterion -- Theorem \ref{theo:HypGauss} -- for a globally bounded hypergeometric series with rational parameters to satisfy Gauss congruences. This enables us to confirm a prediction attributed to Golyshev by Zagier in \cite[p. 757]{zagier}. 

Finally in Sections \ref{sec:Diag} and \ref{sec:Paths}, we briefly study the case when $\eta$ is a rational function and we apply our criterion to various examples issued from diagonals of multivariate rational fractions and random walks in the quarter plane.

\subsection{An arithmetic criterion \textit{via} Gauss congruences}\label{sec:Crit}

Given a prime number $p$, we set $\mathbb Z_{(p)}:=\{r\in \mathbb Q : v_p(r)\ge 0\}$, where $v_p(r)$ is the $p$-adic valuation of $r$. In other words, $\mathbb{Z}_{(p)}$ is the ring of rational numbers the denominator of which is not divisible by $p$ (recall that $v_p(0)=+\infty$ by convention). We consider the following congruences.
\begin{defi}\label{def;gausscongru}
Let $p$ be a prime number. We say that a sequence $(a_n)_{n\in \mathbb Z}$ of rational numbers satisfies Gauss congruences for the prime $p$ if
$a_{np}-a_n\in np\mathbb Z_{(p)}$ for all integers~$n$.
\end{defi}

This is equivalent to 
$a_{mp^{s+1}}-a_{mp^s}\in p^{s+1}\mathbb  Z_{(p)}$ for all integers $m\in \mathbb Z$ and $s\ge 0$, a property that often appears in this form in the literature. These congruences hold for an integer sequence $(a_n)_{n\geq 0}$ and for all prime numbers if and only if, for all $n\geq 0$, we have 
\begin{equation}\label{eq:OriginalGauss}
\sum_{d\mid n}\mu(n/d)a_d\equiv 0\mod n,
\end{equation}
where $\mu$ is the Möbius function. The congruence \eqref{eq:OriginalGauss} was first proved by Gauss when $a_n=r^n$ and $r$ is a prime number, and was later generalized to all integers $r\in\mathbb{Z}$. We refer to \cite{Zarelua} for a survey of these congruences and to \cite{BHS18, BV21, Minton} for recent results on Gauss congruences for multivariate rational fractions.

Following \cite{BHS18}, when a sequence of rational numbers $(a_n)_{n\in \mathbb Z}$ satisfies Gauss congruences for almost all prime numbers $p$ (\footnote{We say that a property holds for almost all prime numbers $p$ when it holds for all but finitely many prime numbers $p$.}), we say that it has the \textit{Gauss property}. In this case, we also say that its bilateral generating series  has the Gauss property. In the particular situation that $a_n=0$ for all $n<r$ for some $r\in \mathbb Z$, \textit{i.e.} the generating series is a Laurent series in $\Q((x))$, we shall say that the sequence $(a_n)_{n\ge r}$ satisfies Gauss congruences. However, it is easily shown (see Section \ref{sec:ProofTheo1}) that in this case necessarily $a_n=0$ for negative $n$, or in other words that if $f\in\Q((x))$ has the Gauss property, then $f\in\Q[[x]]$.

Our arithmetic criterion, which is proved in Section \ref{sec:proofthmsuper}, is the following. 

\begin{theo}\label{theo:super}
Let $\eta\in\Q((x))$  be algebraic over $\Q(x)$. Then the differential equation $y'=\eta y$ admits a non-trivial solution $y$ algebraic over $\mathbb Q(x)$ if and only if $x\eta(x)$ has the Gauss property. 
\end{theo}

Let us explain how Theorem \ref{theo:super} yields an arithmetic characterization of Abel's problem for algebraic series that admit a rational Puiseux expansion at some point $\delta\in\mathbb{C}\cup\{\infty\}$. Let $\eta$ be an algebraic function over $\mathbb{C}(x)$. Let $\delta\in\mathbb{C}$, respectively $\delta=\infty$, and consider the Puiseux expansion of $\eta$ at $\delta$  given respectively by 
\begin{equation}\label{eq:expansions}
\sum_{n=r}^\infty p_n (x-\delta)^{n/d}\quad\textup{and}\quad\sum_{n=r}^\infty p_n x^{-n/d},
\end{equation}
where $r$ and $d\geq 1$ are both integers. We call $(p_n)_{n\in\mathbb{Z}}$ the sequence of the coefficients of the Puiseux expansion of $\eta$ at $\delta$, where $p_n=0$ for $n<r$. We say that the Puiseux expansion is \textit{rational} if $p_n\in\mathbb{Q}$ for all $n$. In this case, we say that $\eta$ has the \textit{Gauss property} at $\delta$ if $(p_n)_{n\in\Z}$ has the Gauss property (by the remark preceding Theorem \ref{theo:super}, the latter property implies that $p_n=0$ for $n<0$). This is a generalization of the above definition for Laurent series in $\Q((x))$, when $d=1$ and $\delta=0$.     

A direct consequence of Theorem \ref{theo:super} is the following criterion.

\begin{coro}\label{coro:super}
Let $\eta$ be an algebraic function over $\mathbb C(x)$ which has a rational Puiseux expansion at $\delta\in\mathbb{C}$, respectively at $\delta=\infty$. Then the differential equation $y'=\eta y$ admits a non-trivial solution $y$ algebraic over $\mathbb C(x)$ if and only if $(x-\delta)\eta(x)$, respectively $x\eta(x)$, has the Gauss property at $\delta$. 
\end{coro}

\begin{proof}
We first assume that $\eta$ has a Puiseux expansion at $\delta\in\mathbb{C}$ of the form \eqref{eq:expansions} with $p_n\in\Q$ for all $n$. We make the change of variables $x=t^d+\delta$ which yields the differential equation $\widetilde{y}'(t)=\widetilde{\eta}(t)\widetilde{y}(t)$, with $\widetilde{y}(t):=y(t^d+\delta)$ and $\widetilde{\eta}(t):=dt^{d-1}\eta(t^d+\delta)\in \Q((t))$. Applying Theorem~\ref{theo:super} to $\widetilde{\eta}(t)$, we obtain that the differential equation $y'=\eta y$ admits a non-trivial solution $y$ algebraic over $\mathbb{C}(x)$ if and only if 
$t\widetilde{\eta}(t)=\sum_{n=r}^\infty dp_nt^{n+d}$
has the Gauss property (at $0$). This is equivalent to saying that the sequence $(dp_{n-d})_{n\in\Z}$ has the Gauss property, or equivalently that $(x-\delta)\eta(x)$ has the Gauss property at $\delta$. 

Similarly, when $\delta=\infty$, the change of variables $x=1/t^d$ shows that $y'=\eta y$ has a non-trivial algebraic solution over $\mathbb{C}(x)$ if and only if  $(-dp_{n+d})_{n\in\Z}$ has the Gauss property, which is equivalent to saying that $x\eta(x)$ has the Gauss property at $\infty$.
\end{proof}

It follows from Corollary \ref{coro:super} that the Gauss property is somehow independent of the point where $\eta$ admits a rational Puiseux expansion. If $\eta$ is an algebraic function and $\delta\in\mathbb{C}$, then we write $\eta_\delta(x):=(x-\delta)\eta(x)$ and $\eta_\infty(x)=x\eta(x)$.

\begin{coro}\label{coro:Puiseux}
Let $\eta$ be an algebraic function over $\mathbb C(x)$ and $\delta_1$ and $\delta_2$ two points in $\mathbb{C}\cup\{\infty\}$ at which $\eta$ admits rational Puiseux expansions. Then $\eta_{\delta_1}$ has the Gauss property at $\delta_1$ if and only if $\eta_{\delta_2}$ has the Gauss property at $\delta_2$.
\end{coro}

We shall from now on focus on the case when $\eta$ admits a rational Puiseux expansion at $\delta=0$. For $y'=\eta y$ to have a non-trivial algebraic solution, $\eta$ must admit a Puiseux expansion at $0$ of the form (see Section \ref{sec:ProofTheo1} for details)
\begin{equation}\label{eq:Puiseuxeta}
\eta(x)=\sum_{n=0}^\infty a_nx^{n/d-1}\in\mathbb{C}((x^{1/d})).
\end{equation}
Then the sequence of coefficients of the Puiseux expansion of $x\eta(x)$ at $0$ reads $(a_n)_{n\geq 0}$ and a non-trivial solution of \eqref{eq: EqDiff} is given by 
\begin{equation}\label{eqPuiseuxy}
y(x):=x^{a_0}\exp\bigg(d\sum_{n=1}^\infty \frac{a_{n}}n x^{n/d}\bigg)=:x^{a_0}g(x)\in x^{a_0} \mathbb{C}((x^{1/d})) .
\end{equation}
We shall also use the less precise notation $y:=\exp\int \eta$ with the  meaning of \eqref{eqPuiseuxy}. Note that an algebraic function over $\mathbb C(x)$ which admits a Puiseux expansion at a rational point with rational coefficients is necessarily algebraic over $\mathbb Q(x)$. In particular, $\eta$ as given in \eqref{eq:Puiseuxeta} with $a_n\in\mathbb Q$ is algebraic over $\mathbb Q(x)$, and thus satisfies a linear differential equation over $\mathbb Q(x)$ (this is based on the fact that $\eta'\in \mathbb Q(x,\eta)$, see \cite{comtet}). Similarily, when $y$ given  by \eqref{eqPuiseuxy} is algebraic over $\mathbb C(x)$, it is automatically algebraic over $\mathbb Q(x)$; observe that $y$ is algebraic over $\mathbb{Q}(x)$ if and only if $g$ is too.

\begin{Remarks*} 
Let us make a few remarks on Theorem \ref{theo:super} and its corollaries.

$\bullet$ When $y$ is algebraic over $\mathbb{Q}(x)$, Corollary \ref{coro:super} shows that, for all but finitely many prime numbers $p$, the sequence $(a_n)_{n\geq 0}$ defined by \eqref{eq:Puiseuxeta} satisfies Gauss congruences for $p$. This bound on $p$ will be made more precise in the proof when $d=1$: Gauss congruences hold at least for any $p$ that does not divide the smallest positive integer $\lambda$ such that $g(\lambda x)\in \mathbb Z[[x]]$ (where $g$ is defined in \eqref{eqPuiseuxy} with $d=1$). 

$\bullet$  When $d=1$ in \eqref{eqPuiseuxy}, the assertion $g\in\mathbb{Z}[[x]]$ has a well-known interpretation in formal group theory, see \cite{Beukers, Honda72}.

$\bullet$ When $y$ is algebraic over $\mathbb{Q}(x)$, we have  $\eta\in\mathbb Q(x,y)$ because $y'\in \mathbb Q(x,y)$ (see above) and $\eta=y'/y$. It is also easy to prove that $y$ is algebraic over $\mathbb{Q}(x)$ if and only if there exists a positive integer $m$ such that $m\eta$ is the logarithmic derivative of an element of $\mathbb{Q}(x,\eta)$, which yields $y^m\in\mathbb{Q}(x,\eta)$. A procedure to determine in a finite number of steps whether $m\eta$ is the logarithmic derivative of an element of $\mathbb{Q}(x,\eta)$ for some $m$ is given in \cite[Section~6]{BD}. It would be interesting to bound explicitely $m$ and so the degree of $y$ in terms of $\eta$. The case $\eta\in \mathbb Q(x)$ already shows that it is not always true that $m=1$ works (see Section~\ref{sec:RatFrac}). 
\end{Remarks*}

\subsection{When $\eta(x)$ is a rational fraction}\label{sec:RatFrac}

If $\eta$ belongs to $\Qbar(x)$, then the non-trivial algebraic solutions of $y'=\eta y$ are arithmetic Nilsson-Gevrey series of order $0$ (see \cite{Rivoal73} for the definition which is not essential here). By \cite[Proposition 3]{Rivoal73}, we obtain that
\begin{equation}\label{eq: y et alpha}
y(x)=c\prod_{i\in I}(x-\lambda_i)^{s_i}\quad\textup{and}\quad \eta(x)=\sum_{i\in I}\frac{s_i}{x-\lambda_i},
\end{equation}
where $c\in\Qbar^\ast$, $I$ is a finite set, $\lambda_i\in\Qbar$ and $s_i\in\mathbb{Q}$ for all $i\in I$. 

In our context, we are interested in the particular case $\eta\in\mathbb{Q}(x)$. Applying \eqref{eq: y et alpha} in this case yields the equivalence: there is a non-trivial algebraic solution to $y'=\eta y$ with $\eta\in\mathbb{Q}(x)$ if and only if we have
$$
y(x)=c\prod_{i\in I}u_i(x)^{s_i}\quad\textup{and}\quad \eta(x)=\sum_{i\in I}s_i\frac{u_i'(x)}{u_i(x)},
$$
where $c\in\Qbar^\ast$, $I$ is a finite set, $u_i\in\mathbb{Z}[x]$ and $s_i\in\mathbb{Q}$ for all $i\in I$. Combining this equivalence with our criterion Theorem \ref{theo:super}, we retrieve the following recent result of Minton. 

\begin{thmx}[Minton \cite{Minton}] \label{minton}
A rational fraction $f\in\mathbb{Q}(x)$  has the Gauss property (at 0) if and only if $f(x)$ is a $\mathbb{Q}$-linear combination of terms $xu'(x)/u(x)$ with $u\in\mathbb{Z}[x]$. 
\end{thmx}

Here we implicitly say that $f\in\mathbb{Q}(x)$ has the Gauss property at 0 when the coefficients of its Laurent expansion at the origin $f(x)=\sum_{n\in\mathbb{Z}}a_nx^n$, with $a_n=0$ for large negative integers $n$, defines a sequence $(a_n)_{n\in\mathbb{Z}}$ with the Gauss property.

\subsection{The hypergeometric case and Golyshev's predictions}\label{sec:Hyp}

In this section, we state our results which give a complete resolution of the case when $x\eta(x)$ is an algebraic hypergeometric series with rational parameters. In particular,  such a hypergeometric series is algebraic over $\mathbb Q(x)$. 
\medskip

Let $\boldsymbol{\alpha}:=(\alpha_1,\dots,\alpha_r)$ and $\boldsymbol{\beta}:=(\beta_1,\dots,\beta_s)$ be tuples of rational numbers in $\mathbb{Q}\setminus\mathbb{Z}_{\leq 0}$. The generalized hypergeometric series with rational parameters \cite{slater} is defined by
$$
{}_{r}F_s\left[\begin{array}{c}\alpha_1,\dots,\alpha_r\\ \beta_1,\dots,\beta_s \end{array} ;x\right] := \sum_{n=0}^\infty \frac{(\alpha_1)_n\cdots (\alpha_r)_n}{(\beta_1)_n\cdots (\beta_s)_n} \frac{x^n}{n!},
$$
where $(\alpha)_n$ is the Pochhammer symbol defined by $(\alpha)_n:=\alpha(\alpha+1)\cdots (\alpha+n-1)$ for $n\geq 1$ and $1$ if $n=0$. We discard the case when $x\eta(x)$ is a polynomial, already considered in Section \ref{sec:RatFrac}, by assuming $\alpha_i\notin\mathbb{Z}_{\leq 0}$ for all $i$. We set
$$
\mathcal{Q}_{\boldsymbol{\alpha},\boldsymbol{\beta}}(n):=\frac{(\alpha_1)_n\cdots(\alpha_r)_n}{(\beta_1)_n\cdots(\beta_s)_n}\quad\textup{and}\quad \mathcal{F}_{\boldsymbol{\alpha},\boldsymbol{\beta}}(x):=\sum_{n=0}^\infty\mathcal{Q}_{\boldsymbol{\alpha},\boldsymbol{\beta}}(n)x^n,
$$
so that
$$
\mathcal{F}_{\boldsymbol{\alpha},\boldsymbol{\beta}}(x)= {}_{r+1}F_s\left[\begin{array}{c}\alpha_1,\dots,\alpha_r,1\\ \beta_1,\dots,\beta_s \end{array} ;x\right],
$$
which has a finite positive radius of convergence if and only if $r=s$. Hence, in the rest of this section, we assume that $r=s$ because this is a necessary condition for $\mathcal{F}_{\boldsymbol{\alpha},\boldsymbol{\beta}}(x)$ to be a non-polynomial algebraic function.

We say that $\mathcal{F}_{\boldsymbol{\alpha},\boldsymbol{\beta}}$ is \textit{factorial} if there exists a non-zero rational constant $C$ and tuples of positive integers $\mathbf{e}=(e_1,\dots,e_u)$ and $\mathbf{f}=(f_1,\dots,f_v)$ such that
\begin{equation} \label{eq33}
\mathcal{F}_{\boldsymbol{\alpha},\boldsymbol{\beta}}(Cx) = \sum_{n=0}^\infty \frac{(e_1 n)!(e_2 n)!\cdots (e_u n)!}{(f_1 n)!(f_2 n)!\cdots (f_v n)!} x^n.
\end{equation}
Because $r=s$, we necessarily have $\sum_j e_j =\sum_j f_j$; this property will be implicit below.
We shall denote the series on the right hand side of \eqref{eq33} by $F_{\mathbf{e},\mathbf{f}}(x)$ and we also set
$$
Q_{\mathbf{e},\mathbf{f}}(n):=\frac{(e_1 n)!(e_2 n)!\cdots (e_u n)!}{(f_1 n)!(f_2 n)!\cdots (f_v n)!}.
$$
The series $\mathcal{F}_{\boldsymbol{\alpha},\boldsymbol{\beta}}$ is factorial if and only if $\boldsymbol{\alpha}$ and $\boldsymbol{\beta}$ are $R$-partitioned in the sense of \cite[\S 7]{delaygue1}, see Section \ref{sec:HypProofs} for more details.

According to Zagier \cite[p. 757]{zagier}, Golyshev has predicted, based on an argument about extensions of motives, that 
$$
y_{\mathbf{e},\mathbf{f}}(x):=\exp\int \frac{F_{\mathbf{e},\mathbf{f}}(x)}{x}\mathrm{d}x=x\exp\bigg(\sum_{n=1}^\infty \frac{Q_{\mathbf{e},\mathbf{f}}(n)}{n}x^n\bigg)
$$
is always algebraic over $\mathbb Q(x)$ when $F_{\mathbf{e},\mathbf{f}}$  is itself algebraic over $\mathbb Q(x)$. 

By {\em ad hoc} explicit computations, Zagier has proved in \cite[pp. 757-759]{zagier} that Golyshev's predictions (\footnote{ Zagier has not explicitely proved that $y_{\mathbf{e},\mathbf{f}}\in \mathbb Q(x, F_{\mathbf{e},\mathbf{f}})$, but the indications he has given on the degrees of $y_{\mathbf{e},\mathbf{f}}$ and $F_{\mathbf{e},\mathbf{f}}$ prove this fact in the cases he considered, using the third remark made after Corollary \ref{coro:Puiseux}.}) hold when $Q_{\mathbf{e},\mathbf{f}}(n)$ is one of 
$$
\binom{en}{fn},\quad 
\frac{(6n)!n!}{(3n)!(2n)!^2}\quad\textup{or}\quad
\frac{(10n)!n!}{(5n)!(4n)!(2n)!},
$$
with $e\geq f\geq 1$. Zagier has also mentionned that Bloch had sketched to him a proof when the algebraic curve defined by $F_{\mathbf{e},\mathbf{f}}(x)$ is rational, but that the general case seemed to be open. In passing, let us mention that the algebraicity of the function $y_{\mathbf{e},\mathbf{f}}$  associated to $\binom{en}{fn}$ naturally appears in subjects pertaining to physics; see for instance \cite[\S 1, Eq. (1.2)]{quiver} and \cite[\S 7.5]{kont1}.

By characterizing the globally bounded hypergeometric series with rational parameters having the Gauss property, we prove that the first part of Golyshev's predictions is true and optimal within the class of algebraic hypergeometric series with rational parameters.

\begin{theo} \label{theo:HypCrit} Let $\boldsymbol{\alpha}$ and $\boldsymbol{\beta}$ be tuples of parameters in $\mathbb{Q}\setminus\mathbb{Z}_{\leq 0}$ such that $\mathcal{F}_{\boldsymbol{\alpha},\boldsymbol{\beta}}$ is algebraic over $\mathbb Q(x)$. Then the function
$$
\exp\int \frac{\mathcal{F}_{\boldsymbol{\alpha},\boldsymbol{\beta}}(x)}{x}\mathrm{d}x
$$
is algebraic over $\mathbb Q(x)$ if and only if $\mathcal{F}_{\boldsymbol{\alpha},\boldsymbol{\beta}}$ is factorial.
\end{theo}

\begin{Remarks*}
Let us make two comments on the other predictions made by Golyshev.

$\bullet$ In the algebraic and factorial case, by the remarks that follow Corollary \ref{coro:Puiseux}, $F_{\mathbf{e},\mathbf{f}}\in \mathbb Q(x, y_{\mathbf{e},\mathbf{f}})$ and there exists a positive integer $m$ such that $y_{\mathbf{e},\mathbf{f}}^m\in \mathbb Q(x, F_{\mathbf{e},\mathbf{f}})$. According to Zagier, Golyshev has predicted that $m=1$ and that $y_{\mathbf{e},\mathbf{f}}$ is a unit over $\mathbb{Z}[1/x]$. In particular, it would follow that $y_{\mathbf{e},\mathbf{f}}$ and $F_{\mathbf{e},\mathbf{f}}$ are algebraic functions of the same degree over $\mathbb{Q}(x)$.

$\bullet$ We observe that in very similar but not exactly hypergeometric situations, none would be true. Consider for instance $f(x)=1+\frac{x}{2(1-x)}$: 
we have 
$$
\exp\int \frac{f(x)}{x}\mathrm{d}x=\frac{x}{\sqrt{1-x}} \notin \mathbb Q(x,f)=\mathbb Q(x)
$$ 
and it is not an integer over $\mathbb Z[1/x]$ because its minimal polynomial over $\mathbb Z[1/x]$ is $(1/x^2-1/x)T^2-1$. 
This example shows in particular that if the remaining parts of Golyshev's predictions are true, their proof could not be based only on the fact that the analytic continuations of algebraic hypergeometric series have exactly one finite singularity.
\end{Remarks*}
\medskip

When $\mathcal{F}_{\boldsymbol{\alpha},\boldsymbol{\beta}}$ is algebraic over $\mathbb{Q}(x)$ but not factorial, Theorem \ref{theo:HypCrit} shows that
$$
y_{\boldsymbol{\alpha},\boldsymbol{\beta}}:=\exp\int\frac{\mathcal{F}_{\boldsymbol{\alpha},\boldsymbol{\beta}}(x)}{x}\mathrm{d}x
$$
is transcendental over $\mathbb{Q}(x)$. But it turns out that the interlacing criterion of Beukers and Heckman \cite{bh} naturally produces algebraic hypergeometric series associated with $\mathcal{F}_{\boldsymbol{\alpha},\boldsymbol{\beta}}$ whose product of the corresponding series $y_{\boldsymbol{\alpha},\boldsymbol{\beta}}$ is always algebraic.  Let us be more precise. Let $\{\cdot\}$ stand for the fractional part function and consider the slight modification of $\{\cdot\}$ given, for all real numbers $x$, by $\langle x\rangle=\{x\}$ if $x$ is not an integer and $1$ otherwise (instead of $0$). We define $\langle\cdot\rangle$ on tuples component-wise, that is $\langle(\alpha_1,\dots,\alpha_r)\rangle:=(\langle\alpha_1\rangle,\dots,\langle\alpha_r\rangle)$. 

\begin{theo}\label{theo:HypAvg}
Let $\boldsymbol{\alpha}$ and $\boldsymbol{\beta}=(\beta_1,\dots,\beta_r)$, with $\beta_r=1$, be disjoint tuples of rational numbers in $(0,1]$ such that $\mathcal{F}_{\boldsymbol{\alpha},\boldsymbol{\beta}}$ is algebraic over $\mathbb{Q}(x)$. Let $d\ge 1$ be the least common multiple of the exact denominators of the $\alpha_i$'s and $\beta_j$'s. Then the function
\begin{equation}\label{eq:yProduct}
\underset{\gcd(k,d)=1}{\prod_{k=1}^{d}}\exp\int\frac{\mathcal{F}_{\langle k\boldsymbol{\alpha}\rangle,\langle k\boldsymbol{\beta}\rangle}(x)}{x}\mathrm{d}x
\end{equation}
is algebraic over $\mathbb{Q}(x)$.
\end{theo}

\begin{Remarks*}
Let us make few remarks on Theorem \ref{theo:HypAvg}.

$\bullet$ When $\mathcal{F}_{\boldsymbol{\alpha},\boldsymbol{\beta}}$ is factorial, then for every $k\in\{1,\dots,d\}$ coprime to $d$, we have $\langle k\boldsymbol{\alpha}\rangle=\boldsymbol{\alpha}$ and $\langle k\boldsymbol{\beta}\rangle=\boldsymbol{\beta}$ so that each term of the product \eqref{eq:yProduct} is equal to $y_{\boldsymbol{\alpha},\boldsymbol{\beta}}$. It follows that when $\mathcal{F}_{\boldsymbol{\alpha},\boldsymbol{\beta}}$ is algebraic and factorial, then Theorem \ref{theo:HypAvg} also gives the conclusion of Theorem \ref{theo:HypCrit}: $y_{\boldsymbol{\alpha},\boldsymbol{\beta}}$ is algebraic over $\mathbb{Q}(x)$. 

$\bullet$ \textit{A contrario}, when $\mathcal{F}_{\langle\boldsymbol{\alpha}\rangle,\langle\boldsymbol{\beta}\rangle}$ is not factorial, then none of the $\mathcal{F}_{\langle k\boldsymbol{\alpha}\rangle,\langle k\boldsymbol{\beta}\rangle}$'s is factorial. But as explained in Section \ref{sec:HypAbel}, the interlacing criterion of Beukers and Heckman gives that, for every $k\in\{1,\dots,d\}$ coprime to $d$, $f_k(x):=\mathcal{F}_{\langle k\boldsymbol{\alpha}\rangle,\langle k\boldsymbol{\beta}\rangle}(x)/x$ is algebraic over $\mathbb{Q}(x)$. By Theorem~\ref{theo:HypCrit}, each $y_k:=\exp\int f_k$ is transcendental over $\mathbb{Q}(x)$. It turns out, by Theorem~\ref{theo:HypAvg}, that the product of the $y_k$'s is algebraic over $\mathbb{Q}(x)$.
\end{Remarks*}
\medskip

Theorems \ref{theo:HypCrit} and \ref{theo:HypAvg} relie on our criterion Theorem \ref{theo:super} and the following characterization of globally bounded hypergeometric series with rational parameters which have the Gauss property. We recall that a series $f(x)\in \mathbb Q[[x]]$ is said to be globally bounded if it has a non-zero radius of convergence and if there exist $C, D\in \mathbb Q\setminus\{0\}$ such that $D\cdot f(Cx)\in \mathbb Z[[x]]$.

\begin{theo}\label{theo:HypGauss}
Let $\boldsymbol{\alpha}$ and $\boldsymbol{\beta}$ be tuples of parameters in $\mathbb{Q}\setminus\mathbb{Z}_{\leq 0}$ such that $\mathcal{F}_{\boldsymbol{\alpha},\boldsymbol{\beta}}$ is globally bounded. Then $\mathcal{F}_{\boldsymbol{\alpha},\boldsymbol{\beta}}$ has the Gauss property if and only if it is factorial. Furthermore, in the latter case, if $\mathcal{F}_{\boldsymbol{\alpha},\boldsymbol{\beta}}(Cx)=F_{\mathbf{e},\mathbf{f}}(x)$ for some non-zero $C\in\mathbb{Q}$, then $F_{\mathbf{e},\mathbf{f}}$ satisfies Gauss congruences for all prime numbers $p$.
\end{theo}

Let us discuss two examples to illustrate the above theorems. 
\medskip

$\bullet$ The sequence of coefficients
$$
Q_{\mathbf{e},\mathbf{f}}(n)=\frac{(30n)!n!}{(15n)!(10n)!(6n)!}\in\mathbb{Z}
$$
was used by Chebyshev in his work on the distribution of prime numbers. According to Rodriguez-Villegas \cite{frv}, they define an algebraic hypergeometric series $F_{\mathbf{e},\mathbf{f}}$ with a surprisingly large degree over $\mathbb{Q}(x)$: $483,840$. By Theorem \ref{theo:HypCrit}, the function 
$$
\exp\int\frac{F_{\mathbf{e},\mathbf{f}}(x)}{x}\mathrm{d}x
$$
is also algebraic over $\mathbb{Q}(x)$. We do not know its degree over $\mathbb{Q}(x)$ but according to Golyshev's prediction, it is also $483,840$.
\medskip

$\bullet$ Let $\boldsymbol{\alpha}$ and $\boldsymbol{\beta}$ be such that
$$
\mathcal{F}_{\boldsymbol{\alpha},\boldsymbol{\beta}}(x)=\sum_{n=0}^\infty\frac{(1/4)_n(11/12)_n}{(1/2)_nn!}x^n.
$$
The interlacing criterion of Beukers and Heckman amounts to comparing the elements of $\langle k\boldsymbol{\alpha}\rangle$ and $\langle k\boldsymbol{\beta}\rangle$ for every $k\in\{1,\dots 12\}$ coprime to $12$, that is $k=1,5,7$ or $11$. The interlacing condition is easily verified and we obtain four algebraic hypergeometric series: $f_1,f_5,f_7$ and $f_{11}$ which are respectively defined by their Taylor coefficients: 
$$
\frac{(1/4)_n(11/12)_n}{(1/2)_nn!},\quad \frac{(1/4)_n(7/12)_n}{(1/2)_nn!},\quad \frac{(5/12)_n(3/4)_n}{(1/2)_nn!}\quad\textup{and}\quad \frac{(1/12)_n(3/4)_n}{(1/2)_nn!}.
$$
Since none of those series is factorial, Theorem \ref{theo:HypCrit} implies that, for all $i\in \{1,5,7,11\}$, the function
$$
y_i(x)=\exp\int\frac{f_i(x)}{x}\mathrm{d}x
$$
is transcendental over $\mathbb{Q}(x)$. But Theorem \ref{theo:HypAvg} shows that the product $y_1y_5y_7y_{11}$ is algebraic over $\mathbb{Q}(x)$. 
\medskip

\noindent {\bf Acknowledgements.} Both authors have partially been funded 
by the ANR project {\em De Rerum Natura} (ANR-19-CE40-0018) for this research. The first author  
has received funding from the European Research Council (ERC) under the European Union’s Horizon 2020 research and innovation programme under the Grant Agreement No 648132.

\section{Applications of Theorem \ref{theo:super}}

We give 
two further applications of Theorem \ref{theo:super} in 
two different directions. 

First, we show that diagonals of rational fractions provide a wide variety of examples of algebraic series $\eta$ with the Gauss property, then leading, by Theorem \ref{theo:super}, to the algebraicity of $y=\exp\int\eta$.

Secondly, we use Theorem \ref{theo:super} in combination with the work of Banderier and Flajolet \cite{BF} to prove that the numbers of directed two-dimensional bridges satisfy Gauss congruences for all prime $p$.


\subsection{Diagonals of rational fractions}\label{sec:Diag}

We present a way to produce examples of algebraic functions $\eta$ in $\mathbb{Q}((x))$ such that $y$ is also algebraic: when $\eta$ is the diagonal of a bivariate rational fraction. 
To that end, for every positive integers $r$ and $s$, we consider the diagonal operator $\Delta_{r,s}$ defined for every bivariate series $f(x,z)=\sum_{n,m\geq -1}a_{n,m}x^nz^m\in(xz)^{-1}\mathbb{C}[[x,z]]$ by
$$
\Delta_{r,s}(f)(x):=\sum_{n=-1}^\infty a_{rn,sn}x^n,
$$
where $a_{rn,sn}=0$ if either $rn<-1$ or $sn<-1$. When $r=s=1$, the operator preserves only the coefficients on the main diagonal and we simply write $\Delta$ for $\Delta_{1,1}$.
\medskip

By a result of Furstenberg \cite{Furstenberg}, diagonals of bivariate rational fractions are algebraic. In particular, if $f(x,z)\in\mathbb{Q}(x,z)$ admits an expansion in $(xz)^{-1}\mathbb{Q}[[x,z]]$, then its corresponding diagonal $\Delta(f)$
is algebraic over $\mathbb{Q}(x)$. 

Another related way to produce algebraic series was  found earlier by P\'olya \cite{Polya}: if $\varphi$ and $\psi$ are algebraic functions regular at the origin, then
\begin{equation}\label{eq:PolyaCrit}
f(x,z):=\frac{\psi(z)}{1-x\varphi(z)}
\end{equation}
admits an expansion as a bivariate power series and, for every positive integers $r$ and $s$, $\Delta_{r,s}(f)$ is algebraic. We will use this result in the particular case when both $\varphi$ and $\psi$ are rational functions.
\medskip

Let $f(x,z)$ be a bivariate rational fraction in $(xz)^{-1}\mathbb{Q}[[x,z]]$, and $r,s$ two positive integers such that 
$$
\eta(x):=\Delta_{r,s}(f)(x)=\sum_{n=0}^\infty a_{n}x^{n-1}
$$
is algebraic. By Theorem \ref{theo:super}, $y:=\exp\int\eta$ is also algebraic if and only if $(a_n)_{n\geq 0}$ has the Gauss property. To study when does the latter occur, we follow \cite{BHS18} and extend Definition~\ref{def;gausscongru} of Gauss congruences to several variables.

\begin{defi}
Let $p$ be a prime number. We say that a family $(a_{\mathbf{n}})_{\mathbf{n}\in\mathbb{Z}^k}$ of rational numbers satisfies {\it Gauss congruences} for the prime $p$ if $a_{\mathbf{m}p^{s+1}}-a_{\mathbf{m}p^s}\in p^{s+1}\mathbb{Z}_{(p)}$ for all $\mathbf{m}\in\mathbb{Z}^k$ and all $s\in\mathbb{Z}_{\geq 0}$.
\end{defi}
We retrieve Definition \ref{def;gausscongru} by setting $k=1$. When a family satisfies Gauss congruences for almost all prime $p$, we also say that it has the \textit{Gauss property}. 

We shall consider Gauss congruences for rational fractions below and they have to be understood as follows. As explained in \cite[Section 2]{BHS18}, a rational fraction $f=P/Q$ has Laurent series associated with each vertex of the Newton polytope (\footnote{If $f(\mathbf{x})=\sum_{i=1}^kf_i\mathbf{x}^{\boldsymbol{\alpha}_i}$ is a Laurent polynomial in $x_1,\dots,x_n$ with $f_i\in\mathbb{Q}$ for all $i$, where we use the vector notation $\mathbf{x}^\mathbf{e}=x_1^{e_1}\cdots x_n^{e_n}$, then the Newton polytope $\Delta\subset\mathbb{R}^n$ of $f$ is the convex hull of its support $\{\boldsymbol{\alpha}_i\,:\,f_i\neq 0\}$.}) of $Q$. It is proved in \cite[Proposition 3.4]{BHS18} that one of these Laurent series has the Gauss property if and only if all Laurent series have the Gauss property. In this case, we also say that $f$ has the \textit{Gauss property}.
\medskip

Gauss congruences are stable by any diagonalization $\Delta_{r,s}$. That is, if $(a_{n,m})_{(n,m)\in\mathbb{Z}^2}$ has the Gauss property then it is also the case for $(a_{rn,sn})_{n\in\mathbb{Z}}$. Hence we are interested in Gauss congruences for multivariate rational fractions. The following two recent results produce examples of rational fractions with the Gauss property.

\begin{thmx}[Beukers--Houben--Straub \cite{BHS18}]\label{theo:BHS}
Let $m\leq n$ and let $f_1,\dots,f_m\in\mathbb{Q}(\mathbf{x})$ be non-zero. Then the rational fraction
$$
\frac{x_1\cdots x_m}{f_1\cdots f_m}\det\left(\frac{\partial f_j}{\partial x_i}\right)_{1\leq i,j\leq m}
$$
has the Gauss property.
\end{thmx}

\begin{thmx}[Beukers--Vlasenko \cite{BV21}]\label{theo:BV21} 
Let $f$ be a Laurent polynomial with integer coefficients such that lattice points in its Newton polytope $\Delta\subset\mathbb{R}^n$ are vertices. Let $g$ be a Laurent polynomial with integer coefficients and support in $\Delta$. Then $g/f$ has the Gauss property.
\end{thmx}

Applying Theorems \ref{theo:BHS} and \ref{theo:BV21} with dimension $n=2$ produces bivariate rational fractions with the Gauss property. Then taking their principal diagonal and applying the result of Furstenberg together with Theorem \ref{theo:super}, we obtain algebraic functions $\eta$ such that $y$ is also algebraic. For example, the central Delannoy numbers 
$$
D(n)=\sum_{k=0}^n\binom{n}{k}\binom{n+k}{k},
$$
are the numbers of $\mathfrak{S}$-walks from $(0,0)$ to $(n,n)$ with steps $\mathfrak{S}=\{(1,0),(0,1),(1,1)\}$ (see Section \ref{sec:Paths} for a definition). They are also given by the algebraic diagonal
$$
\sum_{n=0}^\infty D(n)x^n=\Delta\left(\frac{1}{1-x-y-xy}\right)=\frac{1}{\sqrt{1-6x+x^2}}.
$$
By Theorem \ref{theo:BV21}, the sequence $(D(n))_{n\geq 0}$ has the Gauss property, so that $$
y(x)=x\exp\int_0^x\left(\frac{1}{\sqrt{1-6t+t^2}}-1\right)\frac{\mathrm{dt}}{t}
$$
is also algebraic. Of course, this could be checked directly without first proving that $(D(n))_{n\geq 0}$ has the Gauss property, because an antiderivative of the integrand is
$$
-\textup{arctanh}\big((1-3x)/\sqrt{1-6x+x^2}\big)- \log(x).
$$
But this example illustrates the possibilities and is reminiscent of that of directed two-dimensional walks, that we briefly describe below.

\subsection{Directed two-dimensional walks}\label{sec:Paths}

Fix a finite set of vectors $\mathfrak{S}=\{(a_1,b_1),\dots,(a_m,b_m)\}$. A $\mathfrak{S}$-walk is a sequence $v=(v_1,\dots,v_n)$ such that each $v_j\in\mathfrak{S}$. The geometric realization of a walk $(v_1,\dots,v_n)$ is the sequence of points $(P_0,P_1,\dots,P_n)$ such that $P_0=(0,0)$ and $P_j=P_{j-1}+v_j$. The integer $n$ is referred to as the \textit{size} of the walk. In the rest of this section, we assume that $a_1=\cdots=a_m=1$. In this case, $\mathfrak{S}$ is said \textit{simple} and a $\mathfrak{S}$-walk is a directed two-dimensional lattice path. 

A \textit{bridge} is a walk whose end-point $P_n$ lies on the $x$-axis. An \textit{excursion} is a bridge that lies in the quarter plane $\mathbb{Z}_{\geq 0}\times\mathbb{Z}_{\geq 0}$. We set $B(x)$, respectively $E(x)$, the generating function of the number of bridges, respectively of excursions. That is
$$
B(x):=\sum_{n=0}^\infty B_nx^n\quad\textup{and}\quad E(x):=\sum_{n=0}^\infty E_nx^n,
$$
where $B_n$ and $E_n$ are respectively the numbers of bridges and excursions of size $n$ (in this simple case the size is equal to the length of a walk).
\medskip

Banderier and Flajolet \cite{BF} proved that $B$ and $E$ are algebraic series over $\mathbb Q(x)$ satisfying
$$
E'(x)=\frac{B(x)-1}{x}E(x).
$$
By a direct application of Theorem \ref{theo:super}, we obtain that the sequence of the number of bridges satisfies Gauss congruences for all prime $p$.

\begin{prop}
For all prime numbers $p$ and every non-negative integers $m$ and $s$, we have
$$
B_{mp^{s+1}}\equiv B_{mp^s}\mod p^{s+1}.
$$
\end{prop}

Banderier and Flajolet also gave in \cite{BF} an expression for $B$ as the diagonal of a rational fraction. To state this result, we introduce the \textit{characteristic polynomial} of $\mathfrak{S}$ which is the Laurent polynomial $P(z):=\sum_{i=1}^mz^{b_i}$. We write $c$ for the integer such that $z^cP(z)$ is a polynomial with constant term $1$. We assume that $c\geq 1$, otherwise there is no bridge nor excursion. Then  $B$ is the diagonal $\Delta_{1,c}$ of the rational fraction
\begin{equation}\label{eq:RatBridge}
\frac{1}{1-xz^cP(z)}\in\mathbb{Q}[[x,z]].
\end{equation}
The algebraicity of $B$ follows by the above result of P\'olya. But Theorem~\ref{theo:BV21} of Beukers and Vlasenko  does not apply when $3$ or more steps are allowed in $\mathfrak{S}$. In fact, in this case, the Newton polytope of $1-xz^cP(z)$ contains at least one lattice point which is not a vertex.

\section{Proof of the arithmetic criterion}\label{sec:proofthmsuper}

This section is devoted to the proof of Theorem \ref{theo:super}.

\subsection{The Dieudonné-Dwork Lemma}

A fundamental tool for studying arithmetic properties of exponentials is the following.

\begin{lem}[Dieudonn\'e-Dwork] Let  $F(x)\in 1+x\mathbb Q[[x]]$ and $p$ a prime number. Then $F(x)\in 1+x\mathbb Z_{(p)}[[x]]$ if and only if $F(x^p)/F(x)^p \in 1+px\mathbb Z_{(p)}[[x]]$.
\end{lem}

In our context, we aim at studying $g(x)=\exp(s(x))$ where $s(x)\in x\mathbb{Q}[[x]]$ and this lemma yields

\begin{coro}\label{coro:1} Let $s(x)\in x\mathbb Q[[x]]$ and $p$ a prime number. Then $\exp(s(x))\in 1+x\mathbb Z_{(p)}[[x]]$ if and only if $s(x^p)-ps(x)\in px\mathbb Z_{(p)}[[x]]$.
\end{coro}
This lemma is a particular case of \cite[Lemma 1]{Dwork}, another proof of which can be found in~\cite{Beukers}.

\subsection{Proof of Theorem \ref{theo:super}}\label{sec:ProofTheo1}

Let $\eta\in\Q((x))$ be algebraic over $\Q(x)$. We must prove that the differential equation $y'=\eta y$ admits a non-trivial solution $y$ algebraic over $\Q(x)$ if and only if $x\eta(x)$ has the Gauss property. First we show that, without loss of generality, we can assume that $x\eta(x)\in\Q[[x]]$.

On the one hand, if $x\eta(x)$ has the Gauss property, then set 
$$
x\eta(x)=\sum_{n=r}^\infty a_nx^n,
$$
with $r\in\Z$. For all primes $p$ large enough and at least such that $p>|r|$, and all negative integers $n$, we have both $a_{np}-a_n\in np\Z_{(p)}$ and $a_{np}=0$ so that $a_n\in p\Z_{(p)}$. This shows that $a_n=0$ for every negative integer $n$, and thus that $x\eta(x)\in\Q[[x]]$.

On the other hand, assume the differential equation $y'=\eta y$ admits a non-trivial solution $y$ algebraic over $\Q(x)$. By the Newton-Puiseux theorem \cite[p. 68, Proposition 8]{serre}, $y$ admits a Puiseux expansion:
$$
y(x)=\sum_{n=r}^\infty b_nx^{n/d}=x^{r/d}\sum_{n=0}^\infty b_{n+r}x^{n/d}\in\mathbb{C}((x^{1/d})),
$$
for some integers $r$ and $d\geq 1$ and with $b_r$ non-zero. 
We also have
$$
y'(x)=\frac{1}{d}\sum_{n=r}^\infty nb_nx^{n/d-1}=\frac{x^{r/d-1}}{d}\sum_{n=0}^\infty (n+r)b_{n+r}x^{n/d}.
$$
It follows that $\eta=y'/y$ admits a Puiseux expansion of the form
\begin{equation}\label{Puiseuxeta2}
\eta(x)=\sum_{n=0}^\infty a_nx^{n/d-1}\in\mathbb{C}((x^{1/d})).
\end{equation}
Since by assumption $\eta\in\Q((x))$, we can take $d=1$ so that $x\eta(x)\in\Q[[x]]$.
\medskip

In the rest of the proof of Theorem \ref{theo:super}, we thus assume without loss of generality that 
$$
x\eta(x)=\sum_{n=0}^\infty a_nx^n\in\Q[[x]].
$$
It remains to prove the arithmetic criterion for the sequence $(a_n)_{n\geq 0}$. We  write
$$
y(x):=x^{a_0}\exp\bigg(\sum_{n=1}^\infty \frac{a_{n}}n x^{n}\bigg)=:x^{a_0}g(x), 
$$
which is in $x^{a_0}\mathbb{Q}[[x]]$. Obviously, $y$ is algebraic over $\mathbb{Q}(x)$ if and only if $g$ is too. We write 
$$
f(x):=\sum_{n=1}^\infty a_nx^{n-1}\quad\textup{and}\quad F(x):=\sum_{n=1}^\infty\frac{a_n}{n}x^{n},
$$
so that $g=\exp(F)$ and $g'=fg$. Since $\eta$ is algebraic over $\mathbb{Q}(x)$, $f(x)=\eta(x)-a_0/x$ is too.
\medskip

We prove Theorem \ref{theo:super} in two steps. First we prove the theorem when $f$ belongs to $\mathbb{Z}[[x]]$. Then we use this result to prove the general case when $f\in\mathbb{Q}[[x]]$.

\subsubsection{When $f\in\mathbb{Z}[[x]]$}

In this section we assume that the sequence $(a_n)_{n\geq 1}$ is integer-valued and we prove Theorem~\ref{theo:super} in this particular case, that is: $y$ is algebraic over $\mathbb{Q}(x)$ if and only if $(a_n)_{n\geq 0}$ has the Gauss property. 

\begin{proof}[Proof of the if part.]
 We assume that $(a_n)_{n\geq 0}$ has the Gauss property and show that $y$ is algebraic over $\mathbb{Q}(x)$. We prove this implication in two steps. First we show that there exists a positive integer $\lambda$ such that $g(\lambda x)\in\mathbb{Z}[[x]]$. Then we apply a result of Chudnovsky and Chudnovsky to conclude.
\medskip

The sequence $(a_n)_{n\ge 0}$ satisfies Gauss congruences for all prime numbers $p\ge N$, for some integer $N$. By Corollary \ref{coro:1},  the assertion 
$g\in 1+x\mathbb Z_{(p)}[[x]]$ is equivalent to the assertion $F(x^p)-pF(x)\in px\mathbb Z_{(p)}[[x]]$. This in turn is equivalent to the two following assertions together: 
\begin{align}
\forall n\in \mathbb N^*: \, & p\nmid n \Rightarrow \frac{p a_n}{n}\in p\mathbb Z_{(p)}, \label{eq:1bis}
\\
\forall n\in \mathbb N^*: \, &  \frac{a_n-a_{np}}{n}\in p\mathbb Z_{(p)}. \label{eq:2bis}
\end{align}
 In our situation,~\eqref{eq:1bis} holds because $1/n\in \mathbb Z_{(p)}$ when $p\nmid n$ and $a_n\in \mathbb Z$, and \eqref{eq:2bis} holds as well for $p\ge N$ by assumption. Therefore,  $g\in \mathbb{Z}_{(p)}[[x]]$ for all  $p\ge N$.

Moreover, we have
\begin{align*}
    g(x)&= \exp\big(F(x)\big)
    =\sum_{k=0}^\infty \frac1{k!}\left(\sum_{n=1}^\infty \frac{a_n}{n}x^n\right)^k
    =1+\sum_{n=1}^\infty \sigma_n x^n,
\end{align*}
where
$$
\sigma_n:=\sum_{k=1}^n\frac{1}{k!}\sum_{\underset{m_j\geq 1}{m_1+\cdots+m_k=n}}\frac{a_{m_1}a_{m_2}\cdots a_{m_k}}{m_1m_2\cdots m_k}.
$$
We have $m_1m_2\cdots m_k \mid n!$ for all integers $m_1, \ldots, m_k\in\{1, \ldots, n\}$ such that $\sum_{j=1}^k m_j=n$ because
$$
\frac{n!}{m_1m_2\cdots m_k}= \frac{(m_1+\cdots +m_k)!}{m_1!\cdots m_k!}\cdot (m_1-1)!\cdots(m_k-1)! \in \mathbb Z. 
$$
Consequently, for all $n\ge 1$, $n!^2\sigma_n\in\mathbb{Z}$ because $a_m\in \mathbb Z$ for all $m\ge 1$. For all primes $p$ and all $n\ge 1$, we have
$$
v_p\left(n!^2\right)=2\sum_{k=1}^\infty\left \lfloor \frac{n}{p^k}\right\rfloor\le \frac{2n}{p-1}\le 2n, 
$$
and we deduce that 
$
p^{2n}\sigma_n 
\in \mathbb Z_{(p)}.
$
Letting $\lambda:=\prod_{p < N}p^{2}$, it follows that $g(\lambda x)\in \mathbb Z_{(p)}[[x]]$ for all primes $p<N$. Since $g(x)\in \mathbb Z_{(p)}[[x]]$ for all primes $p\ge N$, we conclude that  $g(\lambda x)\in \mathbb Z[[x]]$.  
\medskip

 We are now in position to apply a fundamental result due to Chudnovsky and Chudnovsky in \cite{chud}: if $h\in \mathbb Z[[x/\mu]]$ (for some $\mu\in \mathbb Q^*$) is such that $h'/h$ is algebraic over $\mathbb Q(x)$, then $h$ is itself algebraic over $\mathbb Q(x)$ (see also \cite[\S3]{acl}, \cite[Theorem 4.4]{kassel} as well as a generalization in  \cite[pp.~123--124]{andre} where $h$ is only assumed to be a $G$-function). We can apply this with $h:=g$ and $\mu:=\lambda$ defined above, and deduce that $g$ is algebraic over $\mathbb Q(x)$.
 \end{proof}

\begin{proof}[Proof of the only if part.] We assume that $y$ is algebraic over $\mathbb{Q}(x)$ and we prove that $(a_n)_{n\geq 0}$ has the Gauss property. Since $g\in 1+x\mathbb Q[[x]]$ is algebraic over $\mathbb Q(x)$, there exists a positive integer $\lambda$ such that $g(\lambda x)\in\mathbb Z[[x]]$ by Eisenstein's theorem. Hence, $g\in 1+x\mathbb Z_{(p)}[[x]]$ for any $p$ that does not divide $\lambda$. By Corollary~\ref{coro:1}, this implies that 
$F(x^p)-pF(x)\in px\mathbb Z_{(p)}[[x]]$ for any $p$ that does not divide $\lambda$. Hence, by \eqref{eq:2bis}, we have $a_{np}-a_{n}\in np\mathbb  Z_{(p)}$ for all $n\ge 1$ and any $p$ that does not divide $\lambda$. Then $(a_n)_{n\geq 0}$ has the Gauss property.
\end{proof}

This completes the proof of Theorem \ref{theo:super} in  the particular case $f\in\mathbb{Z}[[x]]$.

\subsubsection{When $f\in\mathbb{Q}[[x]]$}

In this section we assume that $a_n\in\mathbb{Q}$ for $n\geq 0$.
We complete the proof of Theorem \ref{theo:super}, that is: $y$ is algebraic over $\mathbb{Q}(x)$ if and only if $(a_n)_{n\geq 0}$ has the Gauss property. 
\medskip

We first remark that for any integer $r\neq 0$, $f(x):=r/(1-rx)\in \mathbb Z[[x]]$ is such that $g(x)=1/(1-rx)\in \mathbb Z[[x]]$, so that by the proven case of Theorem \ref{theo:super}, the sequence $(r^n)_{n\ge 0}$ satisfies Gauss congruences for all prime $p$; this well-known generalization of Fermat's little theorem will be used below.

\begin{proof}[Proof of Theorem \ref{theo:super}] 
By Eisenstein's theorem,  there exists an integer $\lambda\ge 1$ such that 
$$
\lambda f(\lambda x)=\sum_{n=1}^\infty \lambda^na_nx^{n-1}\in \mathbb Z[[x]] 
$$
and remains algebraic over $\mathbb Q(x)$. We can thus apply the proven case of Theorem~\ref{theo:super} to 
$$
\eta(x):=\frac{a_0}{x}+\lambda f(\lambda x)\quad\textup{and}\quad
y(x):=x^{a_0}\exp\bigg(\sum_{n=1}^\infty \lambda^n\frac{a_n}{n}x^n\bigg). 
$$
It follows that $y$ is algebraic over $\mathbb Q(x)$ if and only if, for all $n\ge 0$ and large enough $p$, we have
$$
\lambda^{np}a_{np}-\lambda^na_n \in np \mathbb Z_{(p)}.
$$

Now, the algebraicity of $y$ is equivalent to that of $x^{a_0}\exp\left(\sum_{n=1}^\infty \frac{a_n}{n}x^n\right)$. Moreover for any $n\in\mathbb{N}$ and any $p$ that does not divide $\lambda$, we have 
\begin{align*}
\lambda^{np}a_{np}-\lambda^na_n&=\lambda^{np}(a_{np}-a_n)+(\lambda^{np}-\lambda^n)a_n\\
&\equiv \lambda^{np}(a_{np}-a_n)\mod np\mathbb{Z}_{(p)},
\end{align*}
because $\lambda^{np}-\lambda^n\in np\mathbb{Z}_{(p)}$ by the remark above with $r=\lambda$ and $a_n\in\mathbb{Z}_{(p)}$ as $\lambda^na_n\in\mathbb{Z}$. Since $\lambda$ is invertible in $\mathbb{Z}_{(p)}$, we obtain the following equivalence.
$$
\lambda^{np}a_{np}-\lambda^na_n \in np \mathbb Z_{(p)} \Longleftrightarrow a_{np}-a_n \in np \mathbb Z_{(p)}.
$$
This completes the proof of Theorem \ref{theo:super}.
\end{proof}

\section{The hypergeometric case}\label{sec:HypProofs}

This section is devoted to the proofs of Theorems \ref{theo:HypCrit} and \ref{theo:HypAvg}. By Theorem \ref{theo:super}, this amounts to studying Gauss congruences for hypergeometric series.
\medskip

Throughout those proofs, we will make an intensive use of the notations of Section \ref{sec:Hyp}. Furthermore, it is well-known that $\mathcal{F}_{\boldsymbol{\alpha},\boldsymbol{\beta}}$ is factorial if and only if $\boldsymbol{\alpha}$ and $\boldsymbol{\beta}$ are tuples of parameters in $\mathbb{Q}\cap(0,1]$ satisfying
$$
\frac{(x-e^{2i\pi \alpha_1})\cdots(x-e^{2i\pi \alpha_r})}{(x-e^{2i\pi \beta_1})\cdots(x-e^{2i\pi \beta_s})}\in\mathbb{Q}(x),
$$
which is equivalent to saying that $\boldsymbol{\alpha}$ and $\boldsymbol{\beta}$ are $R$-partitioned in the sense of \cite[\S 7]{delaygue1}. In this case, there exist tuples of positive integers $\mathbf{e}=(e_1,\dots,e_u)$ and $\mathbf{f}=(f_1,\dots,f_v)$ such that
$$
\frac{(x-e^{2i\pi \alpha_1})\cdots(x-e^{2i\pi \alpha_r})}{(x-e^{2i\pi \beta_1})\cdots(x-e^{2i\pi \beta_s})}=\frac{(x^{e_1}-1)\cdots(x^{e_u}-1)}{(x^{f_1}-1)\cdots(x^{f_v}-1)}.
$$
We write $|\mathbf{e}|:=\sum_{i=1}^ue_i$ and $|\mathbf{f}|:=\sum_{j=1}^vf_j$. Then we have $r-s=|\mathbf{e}|-|\mathbf{f}|$ and $\mathcal{F}_{\boldsymbol{\alpha},\boldsymbol{\beta}}(Cx)=F_{\mathbf{e},\mathbf{f}}(x)$ with
\begin{equation}\label{eq:C}
C:=\frac{e_1^{e_1}\cdots e_u^{e_u}}{f_1^{f_1}\cdots f_v^{f_v}}.
\end{equation}
By Theorem 4 and Section 4.2.2 in \cite{drr}, if $r=s$, then $C$ is also the smallest positive rational number such that $\mathcal{F}_{\boldsymbol{\alpha},\boldsymbol{\beta}}(Cx)\in\mathbb{Z}[[x]]$.

\subsection{Hypergeometric series and Gauss congruences}

This section is devoted to the proof of Theorem \ref{theo:HypGauss}. To that purpose, we need two lemmas on congruences for hypergeometric terms. If $\boldsymbol{\alpha}$ and $\boldsymbol{\beta}$ are tuples of non-zero rational numbers, then we write $d_{\boldsymbol{\alpha},\boldsymbol{\beta}}$ for the least common multiple of the exact denominators of the $\alpha_i$'s and $\beta_j$'s.

\begin{lem}\label{lem:DworkGauss}
Let $\boldsymbol{\alpha}$ and $\boldsymbol{\beta}$ be tuples of parameters in $\mathbb{Q}\setminus\mathbb{Z}_{\leq 0}$ such that $\mathcal{F}_{\boldsymbol{\alpha},\boldsymbol{\beta}}$ is globally bounded. There exists a constant $c_{\boldsymbol{\alpha},\boldsymbol{\beta}}$ such that the following holds. Let $p > c_{\boldsymbol{\alpha},\boldsymbol{\beta}}$ be a prime number and $k\in\{1,\dots,d_{\boldsymbol{\alpha},\boldsymbol{\beta}}\}$ be such that $kp\equiv 1\mod d_{\boldsymbol{\alpha},\boldsymbol{\beta}}$. Then, for all non-negative integers $m$ and $s$, we have
$$
\mathcal{Q}_{\boldsymbol{\alpha},\boldsymbol{\beta}}(mp^{s+1})- \mathcal{Q}_{\langle k\boldsymbol{\alpha}\rangle,\langle k\boldsymbol{\beta}\rangle}(mp^s)\in p^{s+1}\mathbb{Z}_{(p)}.
$$
\end{lem}

\begin{proof} 
We write $d$ for $d_{\boldsymbol{\alpha},\boldsymbol{\beta}}$. Let $p>d$ be fixed. For every $\alpha\in\mathbb{Z}_{(p)}$, there is a unique element $D_p(\alpha)$ in $\mathbb{Z}_{(p)}$ such that
$$
pD_p(\alpha)-\alpha\in\{0,\dots,p-1\}.
$$
The map $\alpha\mapsto D_p(\alpha)$ was used by Dwork in \cite{DworkCycles} (denoted there as $\alpha\mapsto\alpha'$) to study the $p$-adic valuation of Pochhammer symbols. By a result of Dwork \cite[Lemma 1]{DworkCycles} applied with $a=\mu=0$, if $\alpha\in\mathbb{Z}_{(p)}$, then for all non-negative integers $m$ and $s$, we have
\begin{equation}\label{eq:DworkRatio}
\frac{(\alpha)_{mp^{s+1}}}{(D_p(\alpha))_{mp^s}}\in\left((-p)^{p^s}\varepsilon_{p^s}\right)^m\left(1+p^{s+1}\mathbb{Z}_{(p)}\right), 
\end{equation}
where $\varepsilon_k=-1$ if $k=2$, and $\varepsilon_k=1$ otherwise. By \cite[Lemma 23]{drr} applied with $\ell=1$, there exists a constant $c_\alpha$ such that, for $p>c_\alpha$, we have $D_p(\alpha)=\langle k\alpha\rangle$ where $k\in\{1,\dots,d\}$ satisfies $kp\equiv 1\mod d$. Together with Equation~\eqref{eq:DworkRatio}, it follows that there exists a constant $c_{\boldsymbol{\alpha},\boldsymbol{\beta}}>d$ such that, for all $p>c_{\boldsymbol{\alpha},\boldsymbol{\beta}}$, we have
\begin{equation}\label{eq:CongMult}
\frac{\mathcal{Q}_{\boldsymbol{\alpha},\boldsymbol{\beta}}(mp^{s+1})}{\mathcal{Q}_{\langle k\boldsymbol{\alpha}\rangle,\langle k\boldsymbol{\beta}\rangle}(mp^s)}\in 1+p^{s+1}\mathbb{Z}_{(p)},
\end{equation}
where the term with $\varepsilon_{p^s}$ disappeared because $\boldsymbol{\alpha}$ and $\boldsymbol{\beta}$ are tuples of the same length. To finish the proof of this lemma, it suffices to show that, for $m\in\mathbb{Z}_{\geq 0}$ and $p>c_{\boldsymbol{\alpha},\boldsymbol{\beta}}$, we have 
\begin{equation}\label{eq:Qinteger}
\mathcal{Q}_{\langle k\boldsymbol{\alpha}\rangle,\langle k\boldsymbol{\beta}\rangle}(mp^s)\in\mathbb{Z}_{(p)}.
\end{equation}
Let us consider the total order $\preceq$ on $\mathbb{R}$ defined by
$$
x\preceq y \Longleftrightarrow \big(\langle x\rangle < \langle y \rangle \quad\textup{or}\quad (\langle x\rangle=\langle y\rangle\quad\textup{and}\quad x\geq y)\big).
$$
Christol has proved in \cite{Christol} that $\mathcal{F}_{\boldsymbol{\alpha},\boldsymbol{\beta}}$ is globally bounded if and only if $\boldsymbol{\alpha}$ and $\boldsymbol{\beta}$ have the same length and if, for every $a\in\{1,\dots,d\}$ coprime to $d$ and every $x\in\mathbb{R}$, we have
$$
\xi_{\boldsymbol{\alpha},\boldsymbol{\beta}}(a,x):=\#\{1\leq i\leq r\,:\, a\alpha_i\preceq x\}-\#\{1\leq j\leq r\,:\,a\beta_j\preceq x\}\geq 0.
$$

Let $k\in\{1,\dots,d\}$ coprime to $d$ be fixed. By \cite[Proposition 16]{drr}, we also have $d_{\langle k\boldsymbol{\alpha}\rangle,\langle k\boldsymbol{\beta}\rangle}=d$ and, for every $a\in\{1,\dots,d\}$ coprime to $d$ and every $x\in\mathbb{R}$, we have $\xi_{\langle k\boldsymbol{\alpha}\rangle,\langle k\boldsymbol{\beta}\rangle}(a,x)\geq 0$. It follows by Christol's criterion that $\mathcal{F}_{\langle k\boldsymbol{\alpha}\rangle,\langle k\boldsymbol{\beta}\rangle}$ is globally bounded. 
\medskip

Since $\langle k\boldsymbol{\alpha}\rangle$ and $\langle k\boldsymbol{\beta}\rangle$ are tuples of rational numbers in $(0,1]$, \cite[Theorem 4]{drr} shows that $\mathcal{F}_{\langle k\boldsymbol{\alpha}\rangle,\langle k\boldsymbol{\beta}\rangle}$ is a power series with coefficients in $\mathbb{Z}_{(p)}$ for every prime $p>d$. This yields \eqref{eq:Qinteger} which, together with \eqref{eq:CongMult}, gives
$$
\mathcal{Q}_{\boldsymbol{\alpha},\boldsymbol{\beta}}(mp^{s+1})- \mathcal{Q}_{\langle k\boldsymbol{\alpha}\rangle,\langle k\boldsymbol{\beta}\rangle}(mp^s)\in p^{s+1}\mathbb{Z}_{(p)},
$$
as expected.
\end{proof}

Lemma \ref{lem:DworkGauss} will be used to prove that a hypergeometric series has to be factorial to have the Gauss property. But we need the following lemma, proved in \cite[Lemme 10]{delaygue1} (\footnote{As essentially all $p$-adic congruences of this type proved till 2015 (say) in the context of ``integrality of mirror maps'', the proof of Lemma \ref{lem1} uses in particular $p$-adic congruences for Morita $p$-adic Gamma functions proved by Lang \cite[Chapter 14, Section 1, Lemma 1.1]{lang}. As it appeared later, Lang's congruences do not hold in one case ($p=2$ and $s=2$) and a corrected version is given in \cite[\S 4.4]{drr}: a minus sign must be introduced. It turns out that Lemma~\ref{lem1} is still correct when $p=2$ and $s=2$ because this minus sign only contributes a harmless factor $(-1)^{\sum e_j-\sum f_j}=1$ in its proof.}), to prove that those congruences remain valid for $p\leq c_{\boldsymbol{\alpha},\boldsymbol{\beta}}$. 

\begin{lem}[Lemma 10 of \cite{delaygue1}]\label{lem1}
Let $\mathbf{e}$ and $\mathbf{f}$ be tuples of positive integers such that $|\mathbf{e}|=|\mathbf{f}|$. Then, for all prime numbers $p$, all $s\in \mathbb N$, all $c\in\{0,1, \ldots, p^s-1\}$ and all $m\in \mathbb N$, we have
$$
\frac{Q_{\mathbf{e},\mathbf{f}}(c)}{Q_{\mathbf{e},\mathbf{f}}(cp)}\frac{Q_{\mathbf{e},\mathbf{f}}(cp+mp^{s+1})}{Q_{\mathbf{e},\mathbf{f}}(c+mp ^s)}\in 1+p^{s+1}\mathbb Z_{(p)}.
$$
\end{lem}

\begin{Remark*}The proof of this lemma does not require that $Q_{\mathbf{e},\mathbf{f}}(n)\in \mathbb Z$ for all $n\geq 0$. 
\end{Remark*}

\begin{proof}[Proof of Theorem \ref{theo:HypGauss}]
Let $\boldsymbol{\alpha}$ and $\boldsymbol{\beta}$ be tuples in $\Q\setminus\Z_{\leq 0}$ such that $\mathcal{F}_{\boldsymbol{\alpha},\boldsymbol{\beta}}$ is globally bounded. We can assume without loss of generality that the tuples $\boldsymbol{\alpha}$ and $\boldsymbol{\beta}$ are disjoint. We write $d$ for $d_{\boldsymbol{\alpha},\boldsymbol{\beta}}$ and we recall that $\boldsymbol{\alpha}$ and $\boldsymbol{\beta}$ have the same length. 
\medskip

$\bullet$ First, we assume that $\mathcal{F}_{\boldsymbol{\alpha},\boldsymbol{\beta}}$ is factorial. Let $\mathbf{e}$ and $\mathbf{f}$ be tuples of positive integers such that $\mathcal{F}_{\boldsymbol{\alpha},\boldsymbol{\beta}}(Cx)=F_{\mathbf{e},\mathbf{f}}(x)$ with $C$ given by \eqref{eq:C}. By \cite[Section 4.2.2]{drr} and \cite[Theorem~4]{drr}, we have $\mathcal{F}_{\boldsymbol{\alpha},\boldsymbol{\beta}}(Cx)\in\mathbb{Z}[[x]]$ so, for every $n\in\mathbb{Z}_{\geq 0}$, we have $Q_{\mathbf{e},\mathbf{f}}(n)\in\mathbb{Z}$. We also have $|\mathbf{e}|=|\mathbf{f}|$ because $\boldsymbol{\alpha}$ and $\boldsymbol{\beta}$ have the same length so we can apply Lemma \ref{lem1} with $c=0$. Let $m$ and $s$ be non-negative integers and $p$ a prime number. We write $Q$ as a shorthand for $Q_{\mathbf{e},\mathbf{f}}$. We obtain that
$$
\frac{Q(mp^{s+1})}{Q(mp^s)}\in 1+p^{s+1}\mathbb Z_{(p)},
$$
which yields
$$
Q(mp^{s+1})-Q(mp^s)\in p^{s+1}Q(mp^s)\mathbb Z_{(p)}\subset p^{s+1}\mathbb{Z}_{(p)},
$$
because $Q(n)$ is an integer for all $n\geq 0$. It follows that $F_{\mathbf{e},\mathbf{f}}$ satisfy Gauss congruences for all prime $p$.
\medskip

$\bullet$ Conversely, assume that $\mathcal{F}_{\boldsymbol{\alpha},\boldsymbol{\beta}}$ has the Gauss property. By Lemma \ref{lem:DworkGauss}, there exists a constant $c_{\boldsymbol{\alpha},\boldsymbol{\beta}}$ such that every prime $p>c_{\boldsymbol{\alpha},\boldsymbol{\beta}}$ and all non-negative integers $m$ and $s$, we have
\begin{equation}\label{eq:QGauss2}
\mathcal{Q}_{\boldsymbol{\alpha},\boldsymbol{\beta}}(mp^{s+1})- \mathcal{Q}_{\langle k\boldsymbol{\alpha}\rangle,\langle k\boldsymbol{\beta}\rangle}(mp^s)\in p^{s+1}\mathbb{Z}_{(p)},
\end{equation}
where $k\in\{1,\dots,d\}$ satisfies $kp\equiv 1\mod d$. By definition of the Gauss property, there exists a constant $\kappa>c_{\boldsymbol{\alpha},\boldsymbol{\beta}}$ such that, for every prime $p>\kappa$, we also have
\begin{equation}\label{eq:QGauss1}
\mathcal{Q}_{\boldsymbol{\alpha},\boldsymbol{\beta}}(mp^{s+1})- \mathcal{Q}_{\boldsymbol{\alpha},\boldsymbol{\beta}}(mp^s)\in p^{s+1}\mathbb{Z}_{(p)}.
\end{equation}
Let $a$ and $k$ in $\{1,\dots,d\}$ be such that $ka\equiv 1\mod d$. By subtracting Congruences \eqref{eq:QGauss1} and \eqref{eq:QGauss2} with $s=0$, we obtain that, for every prime $p>\kappa$ satisfying $p\equiv a\mod d$ and all non-negative integers $m$, we have
$$
\mathcal{Q}_{\boldsymbol{\alpha},\boldsymbol{\beta}}(m)- \mathcal{Q}_{\langle k\boldsymbol{\alpha}\rangle,\langle k\boldsymbol{\beta}\rangle}(m)\in p\mathbb{Z}_{(p)}.
$$
By Dirichlet's theorem, there are infinitely many prime numbers $p>\kappa$ satisfying $p\equiv a\mod d$ so, for every $m\in\mathbb{Z}_{\geq 0}$, we have
\begin{equation}\label{eq:Qequal}
\mathcal{Q}_{\boldsymbol{\alpha},\boldsymbol{\beta}}(m)=\mathcal{Q}_{\langle k\boldsymbol{\alpha}\rangle,\langle k\boldsymbol{\beta}\rangle}(m),
\end{equation}
this equation being valid for every $k\in\{1,\dots,d\}$ coprime to $d$. 

Since $\boldsymbol{\alpha}$ and $\boldsymbol{\beta}$ are disjoint, by \cite[Proposition 1]{delaygue1} and Equation \eqref{eq:Qequal} we obtain that, for every $k\in\{1,\dots,d\}$ coprime to~$d$, we have $\boldsymbol{\alpha}=\langle k\boldsymbol{\alpha}\rangle$ and $\boldsymbol{\beta}=\langle k\boldsymbol{\beta}\rangle$ up to a permutation within the tuples. In particular, $\boldsymbol{\alpha}$ and $\boldsymbol{\beta}$ are tuples of elements in $(0,1]$ and the polynomials
$$
\prod_{j=1}^r\left(X-e^{2\pi i\alpha_j}\right)\quad\textup{and}\quad \prod_{j=1}^r\left(X-e^{2\pi i\beta_j}\right)
$$
are left invariant by the action of every Galois automorphism $\sigma\in \mathrm{Gal}(\overline{\mathbb{Q}}/\mathbb{Q})$. Hence those polynomials have integer coefficients and $\mathcal{F}_{\boldsymbol{\alpha},\boldsymbol{\beta}}$ is factorial.
\end{proof}

\subsection{Abel's problem for hypergeometric series}\label{sec:HypAbel}

In this section, we prove Theorems \ref{theo:HypCrit} and \ref{theo:HypAvg}. The combination of Theorems \ref{theo:super} and \ref{theo:HypGauss} easily gives Theorem \ref{theo:HypCrit} as follows.

\begin{proof}[Proof of Theorem \ref{theo:HypCrit}]
Let $\boldsymbol{\alpha}$ and $\boldsymbol{\beta}$ be tuples of parameters in $\mathbb{Q}\setminus\mathbb{Z}_{\geq 0}$ such that $\mathcal{F}_{\boldsymbol{\alpha},\boldsymbol{\beta}}$ is algebraic over $\mathbb{Q}(x)$. By Theorem \ref{theo:super}, the function 
$$
\exp\int \frac{\mathcal{F}_{\boldsymbol{\alpha},\boldsymbol{\beta}}(x)}{x}\mathrm{d}x
$$
is algebraic over $\mathbb Q(x)$ if and only if $\mathcal{F}_{\boldsymbol{\alpha},\boldsymbol{\beta}}$ has the Gauss property. Since $\mathcal{F}_{\boldsymbol{\alpha},\boldsymbol{\beta}}$ is algebraic, it is globally bounded by Eisenstein's theorem. Now Theorem \ref{theo:HypGauss} shows that $\mathcal{F}_{\boldsymbol{\alpha},\boldsymbol{\beta}}$ has the Gauss property if and only if it is factorial, which ends the proof of Theorem \ref{theo:HypCrit}.
\end{proof}

Before proving Theorem \ref{theo:HypAvg}, we recall the interlacing criterion of Beukers and Heckman. 

\begin{defi}
Let $a_j=\exp(2\pi i\lambda_j)$ and $b_j=\exp(2\pi i\mu_j)$, $1\leq j\leq r$, be two sets of numbers on the unit circle in $\mathbb{C}$. Suppose $0\leq\lambda_1\leq\cdots\leq\lambda_r<1$ and $0\leq\mu_1\leq\cdots\leq\mu_r<1$. We say that the sets $\{a_1,\dots,a_r\}$ and $\{b_1,\dots,b_r\}$ interlace on the unit circle if  and only if either
$$
\lambda_1<\mu_1<\lambda_2<\mu_2<\cdots<\lambda_r<\mu_r\quad\textup{or}\quad\mu_1<\lambda_1<\mu_2<\lambda_2<\cdots<\mu_r<\lambda_r.
$$
\end{defi}
Let $\boldsymbol{\alpha}=(\alpha_1,\dots,\alpha_r)$ and $\boldsymbol{\beta}=(\beta_1,\dots,\beta_r)$, with $\beta_r=1$, be tuples of rational numbers in $(0,1]$ such that $\alpha_i\neq \beta_j$ for all $i$ and $j$. Let $d$ be the common denominator of the $\alpha_i$'s and $\beta_j$'s and write $a_j:=\exp(2\pi i\alpha_j)$ and $b_j:=\exp(2\pi i\beta_j)$ for all $j$. By the criterion of Beukers and Heckman \cite[Theorem 4.8]{bh}, the hypergeometric series $\mathcal{F}_{\boldsymbol{\alpha},\boldsymbol{\beta}}$ is algebraic over $\mathbb{Q}(x)$ if and only if, for every $k\in\{1,\dots,d\}$ coprime to $d$, the sets $\{a_1^k,\dots,a_r^k\}$ and $\{b_1^k,\dots,b_r^k\}$ interlace on the unit circle.

\begin{proof}[Proof of Theorem \ref{theo:HypAvg}]
Let $\boldsymbol{\alpha}$ and $\boldsymbol{\beta}=(\beta_1,\dots,\beta_r)$, with $\beta_r=1$, be disjoint tuples of rational parameters in  $(0,1]$ such that $\mathcal{F}_{\boldsymbol{\alpha},\boldsymbol{\beta}}$ is algebraic over $\mathbb{Q}(x)$. We write $d$ for $d_{\boldsymbol{\alpha},\boldsymbol{\beta}}$. For every $k\in\{1,\dots,d\}$ coprime to $d$, the tuples $\langle k\boldsymbol{\alpha}\rangle$ and $\langle k\boldsymbol{\beta}\rangle$ are disjoint and, by the above interlacing criterion, $\mathcal{F}_{\langle k\boldsymbol{\alpha}\rangle,\langle k\boldsymbol{\beta}\rangle}$ is algebraic over $\mathbb{Q}(x)$ and \textit{de facto} globally bounded. Write
$$
f:=\underset{\gcd(k,d)=1}{\sum_{k=1}^d}\mathcal{F}_{\langle k\boldsymbol{\alpha}\rangle,\langle k\boldsymbol{\beta}\rangle}\quad\textup{and}\quad \mathcal{Q}(n):=\underset{\gcd(k,d)=1}{\sum_{k=1}^d}\mathcal{Q}_{\langle k\boldsymbol{\alpha}\rangle,\langle k\boldsymbol{\beta}\rangle}(n).
$$
It follows that $f$ is algebraic over $\mathbb{Q}(x)$. By Lemma \ref{lem:DworkGauss} applied with $\langle k\boldsymbol{\alpha}\rangle$ and $\langle k\boldsymbol{\beta}\rangle$ instead of $\boldsymbol{\alpha}$ and $\boldsymbol{\beta}$ respectively, we obtain that, for every large enough prime $p$ and all non-negative integers $m$ and $s$, we have
\begin{equation}\label{eq:2twist}
\mathcal{Q}_{\langle k\boldsymbol{\alpha}\rangle,\langle k\boldsymbol{\beta}\rangle}(mp^{s+1})- \mathcal{Q}_{\langle a\langle k\boldsymbol{\alpha}\rangle\rangle,\langle a\langle k\boldsymbol{\beta}\rangle\rangle}(mp^s)\in p^{s+1}\mathbb{Z}_{(p)},
\end{equation}
where $a\in\{1,\dots,d\}$ is such that $ap\equiv 1\mod d$. Since $\langle a\langle k\boldsymbol{\alpha}\rangle\rangle=\langle ak\boldsymbol{\alpha}\rangle=\langle b\boldsymbol{\alpha}\rangle$ for $b\in\{1,\dots,d\}$ satisfying $ak\equiv b\mod d$, we have
$$
\underset{\gcd(k,d)=1}{\sum_{k=1}^d}\mathcal{Q}_{\langle a\langle k\boldsymbol{\alpha}\rangle\rangle,\langle a\langle k\boldsymbol{\beta}\rangle\rangle}(n)=\underset{\gcd(b,d)=1}{\sum_{b=1}^d}\mathcal{Q}_{\langle b\boldsymbol{\alpha}\rangle,\langle b\boldsymbol{\beta}\rangle}(n)=\mathcal{Q}(n).
$$
Together with \eqref{eq:2twist}, we obtain, that for every large enough prime $p$ and all non-negative integers $m$ and $s$, we have
$$
\mathcal{Q}(mp^{s+1})-\mathcal{Q}(mp^s)\in p^{s+1}\mathbb{Z}_{(p)}.
$$
Hence $f$ has the Gauss property. By Theorem \ref{theo:super}, it follows that 
$$
\exp\int\frac{f(x)}{x}\mathrm{d}x
$$
is algebraic over $\mathbb{Q}(x)$ and Theorem \ref{theo:HypAvg} is proved.
\end{proof}

\bigskip

\noindent \'Eric Delaygue, Institut Camille Jordan, 
Universit\'e Claude Bernard Lyon 1, 
43 boulevard du 11 novembre 1918, 
69622 Villeurbanne cedex, France\\
delaygue (at) math.univ-lyon1.fr

\medskip

\noindent Tanguy Rivoal Institut Fourier, CNRS et Universit\'e Grenoble Alpes, CS 40700, 
38058 Grenoble cedex 9, France\\
tanguy.rivoal (at) univ-grenoble-alpes.fr

\bigskip

\noindent Keywords: Abel's problem, Algebraic functions, Gauss congruences, Hypergeometric series, Puiseux expansions.

\medskip

\noindent MSC2020: 11A07, 33C20; 34A05, 05A15.


\begin{thebibliography}{1}


\bibitem{andre} Y. Andr\'e, {\em $G$-functions and geometry}, Aspects of Mathematics {\bf 13}, Friedr. Vieweg \&
Sohn. xii, 229 p., 1989.

\bibitem{BD} F. Baldassarri, B. Dwork, {On second order linear differential equations with algebraic solutions}, {\em Am. J. Math.} \textbf{101} (1979), 42--76.

\bibitem{BF} C. Banderier, P. Flajolet, {Basic analytic combinatorics of directed lattice paths}, {\em Theoret. Comput. Sci.} \textbf{281} (2002), no. 1--2, 37--80.

\bibitem{Beukers} F. Beukers, { Some congruences for the Ap\'ery numbers}, {\em J. Number Theory} {\bf 21}.2 (1985), 141--155. 

\bibitem{bh} F. Beukers, G. Heckman, {\em Monodromy for the hypergeometric function ${}_nF_{n-1}$}, Invent.
Math. {\bf 95}.2 (1989), 325--354.

\bibitem{BHS18} F. Beukers, M. Houben, A. Straub, { Gauss congruences for rational functions in several variables}, {\em Acta Arith.} {\bf 184}.4 (2018), 341--362.

\bibitem{BV21} F. Beukers, M. Vlasenko, {Dwork crystals I}, {\em Int. Math. Res. Not.} 2021, No. 12, 8807--8844 (2021).

\bibitem{Boulanger} A. Boulanger, { Contribution \`a l'\'etude des \'equations diff\'erentielles lin\'eaires homog\`enes int\'egrables alg\'ebriquement}, {\em J. de l'\'Ec. Pol.} (2) \textbf{4}  (1898), 1--122.

\bibitem{acl} A. Chambert-Loir, {\em Th\'eor\`emes d'alg\'ebricit\'e en g\'eom\'etrie diophantienne}, S\'eminaire Bourbaki : volume 2000/2001, expos\'es 880-893, Astérisque {\bf 282} (2002), Exposé no. 886, 35 p.

\bibitem{Christol} G. Christol, {\em Fonctions hyperg\'eom\'etriques born\'ees}, Groupe de travail d'analyse ultram\'etrique, tome \textbf{14} (1986--1987), exp. 8, 1--16.

\bibitem{chud} D. V. Chudnovsky, G. V. Chudnovsky, {\em Applications of Pad\'e approximations to the Grothendieck conjecture on linear differential equations}, Number theory, Semin. New York 1983-84, Lect. Notes Math. {\bf 1135} (1985), 52--100.

\bibitem{comtet} L. Comtet, {Calcul pratique des coefficients de Taylor d'une fonction algébrique}, {\em L'Enseignement Math\'ematique} {\bf 10} (1964), 267--270.

\bibitem{quiver} C. Cordova, S.-H. Shao, {Counting Trees in Supersymmetric Quantum Mechanics}, {\em Ann. Inst. Henri Poincaré D}, Comb. Phys. Interact. {\bf 5}.1 (2018), 1--60. 

\bibitem{delaygue1} \'E. Delaygue, {Crit\`ere pour l'int\'egralit\'e des coefficients de Taylor des applications miroir}, {\em J. reine angew. Math.} {\bf 662} (2012), 205--252.

\bibitem{drr} \'E. Delaygue, T. Rivoal, J. Roques, {On Dwork's p-adic formal congruences theorem and hypergeometric mirror maps}, {\em Mem. Amer. Math. Soc.} {\bf 246}, no. 1163 (2017), 100 pages.

\bibitem{Dwork} B. Dwork, {Norm residue symbol in local number fields}, {\em Abh. Math. Sem. Univ. Hamburg} {\bf 22} (1958), 180--190. 

\bibitem{DworkCycles} B. Dwork, {$p$-adic cycles}, {\em Publ. Math  Inst. Hautes \'Etudes Sci.} \textbf{37} (1969), 27--115.

\bibitem{Rivoal73} S. Fischler, T. Rivoal {A note on $G$-operators of order $2$}, preprint 2021, available at {\tt https://hal.archives-ouvertes.fr/hal-03065680}, to appear in Colloquium Mathematicum.

\bibitem{Furstenberg} H. Furstenberg, {Algebraic functions over finite fields}, {\em J. Algebra} {\bf 7} (1967), 271--277.

\bibitem{Honda72} T. Honda, {Formal groups obtained from generalized hypergeometric functions}, {\em Osaka J. Math.} \textbf{9} (1972), 447--462.

\bibitem{kassel} C. Kassel, C. Reutenauer, 
Algebraicity of the zeta function associated to a matrix over a free group algebra.
{\em Algebra Number Theory} {\bf 8}.2 (2014), 497-–511.

\bibitem{kont1} M.~Kontsevich,~Y.~Soibelman, {\em Stability~structures,~motivic~Donaldson-Tho\-mas invariants and cluster transformations}, preprint 2008, available at {\tt https://arxiv.org/abs/0811.2435} 

\bibitem{lang} S. Lang, {\em Cyclotomic fields, I, II}, Combined 2nd edition, vol. {\bf 121}, Graduate Texts in Math., Springer-Verlag, New York, 1990.

\bibitem{Minton} G. T. Minton, {\em Linear recurrence sequences satisfying congruence conditions}, Proc. Amer. Math. Soc. {\bf 142}.7 (2014),  2337--2352.

\bibitem{Polya} G. P\'olya, {Sur les s\'eries enti\`eres dont la somme est une fonction alg\'ebrique}, {\em Enseign. Math.} \textbf{1--2} (1921--1922), 38--47.

\bibitem{Risch} R. H. Risch, {The problem of integration in finite terms}, {\em Trans. Am. Math. Soc.} \textbf{139} (1969), 167--189.

\bibitem{frv} F. Rodriguez-Villegas, {Integral ratios of factorials and algebraic hypergeometric functions}, preprint 2007, 
available at {\tt https://arxiv.org/abs/math/0701362}

\bibitem{serre} J.-P. Serre, {\em Local Fields}, { Graduate Texts in Mathematics} {\bf 67}, Springer-Verlag New York Berlin Heidelberg, 1979.  

\bibitem{Singer} M. F. Singer, {\em Algebraic solutions of $n$-th order linear differential equations}, Proc. Queen's Number Theory Conf. 1979, Queen's Pap. Pure Appl. Math. \textbf{54}  (1980), 379--420.

\bibitem{slater} L. J. Slater, {\em Generalized Hypergeometric Functions}, Cambridge, Cambridge Univ. Press, second edition, 2008.

\bibitem{zagier} D. Zagier, {\em The arithmetic and topology of differential equations}, 
in Proceedings of the European Congress of Mathematics, Berlin, 18-22 July, 2016, Mehrmann, V.; Skutella, M. (Eds.), European Mathematical Society (2018), 717--776.

\bibitem{Zarelua} A. V. Zarelua, {On congruences for the traces of powers of some matrices}, {\em Proc. Steklov Inst. Math.} \textbf{263}.1 (2008), 78--98.

\end{thebibliography}
\end{document}